\documentclass[12pt, two]{article}
\usepackage{}

\usepackage[leqno]{amsmath}
\usepackage{amssymb}
\usepackage{amsthm}
\usepackage{enumerate}
\usepackage{cases}
\usepackage[center]{titlesec}
\usepackage{fancyhdr}
\usepackage[top=1in, bottom=1in, left=1.25in, right=1.25in]{geometry}
\numberwithin{equation}{section}
\usepackage{CJK}
\allowdisplaybreaks

\newtheorem{theorem}{Theorem}[section]
\newtheorem{lemma}{Lemma}[section]
\newtheorem{proposition}{Proposition}[section]
\newtheorem{corollary}{Corollary}[section]

\newtheorem{definition}{Definition}[section]
\numberwithin{equation}{section} 

\title{
Liouville Theorems for critical points of the $p$-Ginzburg-Landau type functional\author{Tian Chong, Bofeng Cheng, Yuxin Dong and Wei Zhang
}
\footnotetext[0]{ This research was supported by NSFC grant No. 11271071 and LMNS, Fudan; the key project}
\footnotetext[0]{grant No. XXKPY1604 , Shanghai Second Polytechnic University}
}

\author{}
\date{}
\begin{document}


\maketitle

\begin{abstract}
\noindent \textbf{Abstract.}
In this paper, we consider the smooth map from a Riemannian manifold to the standard Euclidean space and the $p$-Ginzburg-Landau energy($p\geq1$). Under suitable curvature conditions on the domain manifold, some Liouville type theorems are established by assuming
either growth conditions of the $p$-Ginzburg-Landau energy or an asymptotic condition
at the infinity for the maps. In the end of paper, we obtain the unique constant solution of the constant Dirichlet boundary value problems on starlike domains.
\end{abstract}

\section{Introduction}

\ \ \ \ One of the important problems for harmonic maps or generalized harmonic maps is to study their Liouville type results. (cf. \cite{Ch,GRSB,Hi,SY,Jin}). It is well known that the stress-energy tensor is a useful tool to investigate the energy behavior and some vanishing results of related energy functional. Most Liouville results have been established by assuming either the finiteness of the energy of the map or the smallness of the whole image of the domain manifold under the map. In \cite{Jin}, Z.R. Jin has shown several interesting Liouville theorems for harmonic maps from complete manifolds with assumptions on the asymptotic behavior of the maps at infinity.\\
\indent Let $\Omega\subseteq R^2$ be a smooth bounded simply connected domain. Consider the following functional defined for maps
$u\in H^1(\Omega,\mathbb{C})$:

\begin{align}
E_{\epsilon}(u)=\frac{1}{2}\int_{\Omega}|\nabla u|^2+\frac{1}{4\epsilon^2}\int_{\Omega}(|u|^2-1)^2.\nonumber
\end{align}
Ginzburg-Landau introduced this Functional in the
study of phase transition problems and it plays an important role ever since, especially in superconductivity, superfluidity and XY-magnetism (see details for \cite{KT,Nelson,Saint}). A lot of papers devote to the asymptotic behavior of minimizers $u_{\epsilon}$ of $E_{\epsilon}(u,\Omega)$ in $H$ as $\epsilon\rightarrow 0$. It was shown in those cases that $u_{\epsilon}$ converges strongly to a harmonic map $u_0$ on any compact subset away from the zeros. Readers can refer to \cite{BBH,BBH2,BMR,Struwe} for the progress in this field. In the past decades, $p$-Ginzburg-Landau functionals have been introduced. In \cite{Hong,Lei}, the authors investigated the convergence of a p-Ginzburg-Landau type functional when the parameter goes to zero.\\
\indent In this paper, we consider a smooth map $u:(M^m,g)\rightarrow (R^n,h)$ from a Riemannian manifold to the standard Euclidean space and the following $p$-Ginzburg-Landau energy
\begin{align}
E_{GL}^p(u)=\int_M\frac{|du|^p}{p}+\frac{1}{4\epsilon^n}(1-|u|^2)^2dv_g, \nonumber
\end{align}
where $p\geq 1$ and $\epsilon$ is any small positive number. To generalize the Liouville type results for harmonic maps to the critical points of $p$-Ginzburg-Landau energy functional, we introduce the stress energy tensor $S_{u,p}^{GL}$ associated with the $p$-Ginzburg-Landau functional $E_{GL}^p(u)$. It is easy to show that any critical point of the $p$-Ginzburg-Landau functional satisfies the conservation law, that is, $div S_{u,p}^{GL}=0$. Using a basic integral formula
linked naturally to the conservation law enables us to establish some monotonicity formulae for these critical points of the $p$-Ginzburg-Landau energy functional. Consequently, several Liouville type results can be deduced from these monotonicity formulae under suitable growth conditions on the energy. We also build a Liouville type result under the condition of slowly divergent energy.

Next we want to generalize Jin's results in \cite{Jin} to the critical point of the $p$-Ginzburg-Landau energy functional. The methods we use for proving the result is very similar to Jin's. Firstly, we may use the stress-energy tensor to establish the
monotonicity formula which gives a lower bound for the growth rates of the energy. Secondly, we use the asymptotic assumption of the map at infinity
to obtain the upper energy growth rates. Under suitable
conditions on $u$ and the Hessian of the distance functions of the domain manifolds, one
may show that these two growth rates are contradictory unless the critical point
is constant. In this way, we establish some Liouville theorems for the critical points of the $p$-Ginzburg-Landau energy functional with the asymptotic property at infinity from some complete manifolds.\\
\indent In addition to establishing Liouville type results, the
monotonicity formulae may be used to investigate the constant Dirichlet boundary value problem as well. We obtain the unique constant solution of the constant Dirichlet boundary value problem on starlike domains for the critical point of $p$-Ginzburg-Landau energy functional.\\
\section{$p$-Ginzburg-Landau energy functional and stress-energy tensor}
Let $u:(M^m,g)\rightarrow (R^n,h)$ be a smooth map from a Riemannian manifold to the standard Euclidean space. We consider the following $p$-Ginzburg-Landau energy
\begin{align}
E_{GL}^p(u)=\int_M\frac{|du|^p}{p}+\frac{1}{4\epsilon^n}(1-|u|^2)^2dv_g, \nonumber
\end{align}
where $p\geq 1$ and $\epsilon$ is any small positive number.
Let $\{u_{t}\}(|t|<\kappa)$ with $u_{0}=u$ and $v= \frac{\partial u_t}{\partial t}|_{t=0}$ be a one parameter variation, we have the following lemma.
\begin{lemma}
The first variation formula for $p$-Ginzburg-Landau energy functional
\begin{align}
\frac{d}{dt}|_{t=0}E_{GL}^p(u_t)=-\int_M\langle div(|du|^{p-2}du)+\frac{1}{\epsilon^n}(1-|u|^2)u,v \rangle dv_g.\nonumber
\end{align}
\end{lemma}
\begin{proof}
Let $\{e_i\}_{i=1}^{m}$ be a local orthonormal frame of $TM$. Since the target manifold is Standard Euclidean space, we can perform the following calculations
\begin{align}
\frac{d}{dt}|_{t=0}E_{GL}^p(u_t)&=\int_M\frac{\partial}{\partial t}|_{t=0}(\frac{|du_t|^p}{p})dv_g
+\int_M\frac{\partial}{\partial t}|_{t=0}[\frac{1}{4\epsilon^n}(1-|u_t|^2)^2]dv_g \nonumber \\
&=\int_M|du|^{p-2}\sum_{i=1}^m\langle \nabla_{\frac{\partial}{\partial t}}du_t(e_i),du_t(e_i) \rangle|_{t=0}dv_g-\frac{1}{2\epsilon^n}\int_M(1-|u_t|^2)\frac{\partial}{\partial t}|_{t=0}|u_t|^2dv_g \nonumber \\
&=\int_M|du|^{p-2}\sum_{i=1}^m\langle \nabla_{e_i}du_t(\frac{\partial}{\partial t}),du_t(e_i) \rangle|_{t=0}dv_g-\frac{1}{\epsilon^n}\int_M(1-|u|^2)v\cdot udv_g \nonumber \\
&=\int_M|du|^{p-2}\sum_{i=1}^m\langle \nabla_{e_i}v,du(e_i) \rangle dv_g-\frac{1}{\epsilon^n}\int_M(1-|u|^2)v\cdot udv_g \nonumber \\
&=-\int_M\langle div(|du|^{p-2}du),v \rangle dv_g-\frac{1}{\epsilon^n}\int_M(1-|u|^2)v\cdot udv_g \nonumber \\
&=-\int_M\langle div(|du|^{p-2}du)+\frac{1}{\epsilon^n}(1-|u|^2)u,v \rangle dv_g. \nonumber
\end{align}
\end{proof}
\begin{definition}
$u$ is called a critical point of $p$-Ginzburg-Landau energy functional if
\begin{align}
div(|du|^{p-2}du)+\frac{1}{\epsilon^n}(1-|u|^2)u=0.   \nonumber
\end{align}
When $p=2$, above equation is reduced to $\Delta u+\frac{1}{\epsilon^n}(1-|u|^2)u=0.$
\end{definition}
In [BE], Baird-Eells introduced the stress-energy tensor associated with the usual energy and proved that harmonic maps satisfy the conservation law.
We can also define the stress-energy tensor $S_{u,p}^{GL}$ associated with the $p$-Ginzburg-Landau energy functional $E_{GL}^p(u)$ and prove that the critical points satisfy the conservation law, i.e. $divS_{u,p}^{GL}=0$.
\begin{definition}
Let $u:(M^m,g)\rightarrow (R^n,h)$ be a smooth map from a Riemannian manifold to the standard Euclidean space. The stress-energy tensor of $u$ is the symmetric $2$-tensor on $M$ given by
\begin{align}
S_{u,p}^{GL}=[\frac{|du|^p}{p}+\frac{1}{4\epsilon^n}(1-|u|^2)^2]g-|du|^{p-2}u^*h. \nonumber
\end{align}
\end{definition}

\begin{theorem}
$div S_{u,p}^{GL}(X)=-\langle div(|du|^{p-2}du)+\frac{1}{\epsilon^n}(1-|u|^2)u,du(X)\rangle$, for any $X\in \Gamma(TM)$.
\end{theorem}
\begin{proof}
\indent For any 2-tensor field $W\in\Gamma(T^*M\otimes T^*M)$, the divergence of $W$ is defined by
\begin{align}
(divW)(X)=\sum_{i=1}^m(\nabla^M_{e_i}W)(e_i,X),\label{1033}
\end{align}
where $\{e_i\}_{i=1}^m$ is an local orthonormal basis of $M$. Then we have
\begin{align}
divS_{u,p}^{GL}(X)&=\sum_{i=1}^m[\nabla_{e_i}S_{F,u}^{GL}(e_i,X)]-S_{F,u}^{GL}(e_i,\nabla_{e_i}X) \nonumber \\
&=\sum_{i=1}^me_i\{[\frac{|du|^p}{p}+\frac{1}{4\epsilon^n}(1-|u|^2)^2]\langle e_i,X \rangle\}
-\sum_{i=1}^me_i\{|du|^{p-2}\langle du(e_i),du(X) \rangle\} \nonumber \\
&-[\frac{|du|^p}{p}+\frac{1}{4\epsilon^n}(1-|u|^2)^2]\sum_{i=1}^m\langle e_i,\nabla_{e_i}X \rangle+|du|^{p-2}\sum_{i=1}^m\langle du(e_i),du(\nabla_{e_i}X) \rangle \nonumber \\
&=X[\frac{|du|^p}{p}+\frac{1}{4\epsilon^n}(1-|u|^2)^2]-\sum_{i=1}^me_i(|du|^{p-2})\langle du(e_i),du(X) \rangle \nonumber \\
&-\sum_{i=1}^m|du|^{p-2}\langle \nabla_{e_i}du(e_i),du(X) \rangle-\sum_{i=1}^m|du|^{p-2}\langle du(e_i),\nabla_{e_i}du(X) \rangle \nonumber \\
&+|du|^{p-2}\sum_{i=1}^m\langle du(e_i),du(\nabla_{e_i}X) \rangle \nonumber \\
&=-\langle div(|du|^{p-2}du)+\frac{1}{\epsilon^n}(1-|u|^2)u,du(X) \rangle \nonumber
\end{align}
\end{proof}
\begin{definition}
We say that $u$ satisfies the conservation law if $divS_{u,p}^{GL}=0$.
\end{definition}
\begin{corollary}\label{1034}
If $u:(M^m,g)\rightarrow (R^n,h)$ is a critical point of $p$-Ginzburg-Landau energy functional, then $u$ satisfies the conservation law, i.e., $divS_{u,p}^{GL}=0$.
\end{corollary}

\indent For any vector field $X\in\Gamma(TM)$, let $\theta_{X}$ denote the dual one form of X, that is,
\begin{align}
\theta_X(Y)=g(X,Y),     \ \ \ \ \ \ \ \ \ \forall Y\in\Gamma(TM).
\end{align}
The covariant derivative of $\theta_{X}$ is given by
\begin{align}
(\nabla^M\theta_X)(Y,Z)=(\nabla^M_Y\theta_X)(Z)=g(\nabla^M_Y X,Z), \nonumber
\end{align}
for any $X,Y,Z \in \Gamma(TM)$. If $X=\nabla^M \psi$ is the gradient of some smooth function $\psi$ on $M$, then $\theta_X=d\psi$ and $\nabla^M\theta_X=Hess_g(\psi)$. \\

\indent Let $W\in\Gamma(T^*M \otimes T^*M)$ be any symmetric $2$-tensor. By a direct computation, we have
\begin{align}
div(i_XW)&=(divW)(X)+\langle W,\nabla^M\theta_X\rangle \nonumber \\
&=(divW)(X)+\frac{1}{2}\langle W,L_Xg\rangle, \nonumber
\end{align}
where $i_XW\in A^1(M)$ denotes the interior product by any $X \in \Gamma(TM)$.\\
\indent In terms of the Stoke's formula we get
\begin{lemma}\label{201}
Let $D$ be any bounded domain of $M$ with $C^1$ boundary. Denote by $\nu$ the unit outward normal vector field along $\partial D$. For any symmetric $2$-tensor $W\in\Gamma(T^*M \otimes T^*M)$ and any vector field $X\in\Gamma(TM)$, we have
\begin{align}
\int_{\partial D}(i_XW)(\nu)ds_g=\int_D\langle W,\nabla^M\theta_X\rangle+(divW)(X)dv_g \nonumber
\end{align}
and
\begin{align}
\int_{\partial D}(i_XW)(\nu)ds_g=\int_D\frac{1}{2}\langle W,L_Xg\rangle+(divW)(X)dv_g. \nonumber
\end{align}
\end{lemma}
Applying Lemma \ref{201} to $S_{u,p}^{GL}$, we immediately obtain the following integral formulae:
\begin{align}
\int_{\partial D}S_{u,p}^{GL}(X,\nu)ds_g=\int_D\langle S_{u,p}^{GL},\nabla \theta_X\rangle+(divS_{u,p}^{GL})(X) dv_g \nonumber
\end{align}
and
\begin{align}
\int_{\partial D}S_{u,p}^{GL}(X,\nu)ds_g=\int_D\frac{1}{2}\langle S_{u,p}^{GL},L_Xg\rangle+(divS_{u,p}^{GL})(X) dv_g. \nonumber
\end{align}
If $u$ is a critical point of $p$-Ginzburg-Landau energy functional, by Corollary \ref{1034} we obtain
\begin{align}
\int_{\partial D}S_{u,p}^{GL}(X,\nu)ds_g&=\int_D\langle S_{u,p}^{GL},\nabla \theta_X\rangle dv_g.  \nonumber \\
&=\int_D\frac{1}{2}\langle S_{u,p}^{GL},L_Xg\rangle dv_g. \label{1006}
\end{align}
\indent For applications of the stress-energy tensor, the readers may refer to \cite{DL,DLY,DW,Xin}. In next section, we will use the similar method to establish the monotonicity formulae.

\section{Monotonicity formulae and Liouville type results under growth
conditions.}
\ \ \ \ From now on, we always assume that ($M^m,g$) is a complete Riemannian manifold with a pole $x_0$. A pole $x_0 \in M$ is a point such that the exponential map from the tangent space to $M$ at $x_0$ into $M$ is a diffeomorphism.  We will establish monotonicity formulae on these mainfolds.\\
\indent Denote by $r(x)$ the distance function relative to the pole $x_0$. For any $x\in M$, let $\lambda_1(x)\leq \lambda_2(x)\leq \cdots \leq \lambda_m(x)$ be the eigenvalues of $Hess_g(r^2)$ at $x$.
\begin{theorem}\label{1014}
Assume that $u:(M^m,g)\rightarrow (R^n,h)$ is the critical point of $p$-Ginzburg-Landau energy functional.
 If there exists a constant $\sigma>0$ such that \begin{align}
(P_1)\ \ \ \ \ \ \ \frac{1}{2}(\sum_{i=1}^m\lambda_i-p\lambda_m)\geq \sigma, \nonumber
\end{align}
then
\begin{align}
\frac{\int_{B_{\rho_1}(x_0)}\frac{|du|^p}{p}+\dfrac{1}{4\epsilon^n}(1-|u|^2)^2 dv_g }{\rho_1^\sigma}\leq 
\frac{\int_{B_{\rho_2}(x_0)}\frac{|du|^p}{p}+\dfrac{1}{4\epsilon^n}(1-|u|^2)^2 dv_g }{\rho_2^\sigma}, \nonumber
\end{align}
for any $0<\rho_1\leq \rho_2$.
\end{theorem}
\begin{proof}
Set $D=B_R(x_0)=\{x\in M|r(x)\leq R\}$ and $X=\frac{1}{2}\nabla r^2=r\frac{\partial}{\partial r}$. Since $u$ is a critical point, use (\ref{1006}) we have
\begin{align}
\int_{\partial B_R(x_0)}S_{u,p}^{GL}(X,\nu)ds_g=\int_{B_R(x_0)}\frac{1}{2}\langle S_{u,p}^{GL},L_Xg\rangle dv_g, \label{1009}
\end{align}
where $\nu=\frac{\partial}{\partial r}$ is the unit outward normal vector field of $B_R(x_0)$.\\
\indent Let $\{e_i\}_{i=1}^m$ be an orthonormal frame of ($M,g$). Moreover, we can assume that $Hess_{g}(r^2)$ becomes a diagonal matrix with respect to $\{e_i\}_{i=1}^m$.
\begin{align}
&\langle S_{u,p}^{GL},\frac{1}{2}L_X g\rangle
=\frac{1}{2}\langle S_{u,p}^{GL},Hess_g(r^2)\rangle  \nonumber \\
&=\frac{1}{2}(\dfrac{|du|^p}{p}+\dfrac{1}{4\epsilon^n}(1-|u|^2)^2)\langle g ,Hess_g(r^2)\rangle
-\frac{1}{2}|du|^{p-2}\langle u^{\ast}h,Hess_g(r^2)\rangle \nonumber \\
&=\frac{1}{2}(\dfrac{|du|^p}{p}+\dfrac{1}{4\epsilon^n}(1-|u|^2)^2)
\sum_{i,j=1}^{m}g(e_i,e_j)\cdot Hess_g(r^2)(e_i,e_j) \nonumber  \\
&\ \ \ \ \ \ \ \ \ \ \ \ \ \ \ \ \ \ \ \ \ \ \  -\frac{1}{2}|du|^{p-2}\sum_{i,j=1}^{m} u^{\ast}h(e_i,e_j)\cdot Hess_g(r^2)(e_i,e_j) \nonumber \\
&\geq \frac{1}{2}(\dfrac{|du|^p}{p}+\dfrac{1}{4\epsilon^n}(1-|u|^2)^2)
\sum_{i=1}^{m}\lambda_i-\frac{1}{2}|du|^{p}\lambda_{m} \nonumber \\
&\geq \frac{1}{2}(\dfrac{|du|^p}{p}+\dfrac{1}{4\epsilon^n}(1-|u|^2)^2)
\sum_{i=1}^{m}\lambda_i-\frac{1}{2}(\frac{|du|^{p}}{p}+\dfrac{1}{4\epsilon^n}(1-|u|^2)^2)p\lambda_{m} \nonumber \\
&=\frac{1}{2}(\dfrac{|du|^p}{p}+\dfrac{1}{4\epsilon^n}(1-|u|^2)^2)(\sum_{i=1}^{m}\lambda_i-p\lambda_{m}). \label{1010}
\end{align}
On the other hand,
\begin{align}
\int_{\partial B_R(x_0)} S_{u,p}^{GL}(r\frac{\partial}{\partial r},\nu)ds_g
&\leq \int_{\partial B_R(x_0)}(\dfrac{|du|^p}{p}+\dfrac{1}{4\epsilon^n}(1-|u|^2)^2)g(r\frac{\partial}{\partial r},\nu)ds_g\nonumber\\
&= R\int_{\partial B_R(x_0)}(\dfrac{|du|^p}{p}+\dfrac{1}{4\epsilon^n}(1-|u|^2)^2)g(\frac{\partial}{\partial r},\frac{\partial}{\partial r})ds_g\nonumber\\
&= R\frac{d}{dR}\int_0^R(\int_{\partial B_r(x_0)}(\dfrac{|du|^p}{p}+\dfrac{1}{4\epsilon^n}(1-|u|^2)^2)ds_g)dr\nonumber\\
&= R\frac{d}{dR}\int_{ B_R(x_0)}(\dfrac{|du|^p}{p}+\dfrac{1}{4\epsilon^n}(1-|u|^2)^2)dv_g.\label{1011}
\end{align}
It follows from (\ref{1009}), (\ref{1010}) and (\ref{1011}) that
\begin{align}
&R\frac{d}{dR}\int_{ B_R(x_0)}(\dfrac{|du|^p}{p}+\dfrac{1}{4\epsilon^n}(1-|u|^2)^2)dv_g \nonumber  \\
&\geq\int_{ B_R(x_0)}\frac{1}{2}(\sum_{i}^{m}\lambda_i-p\lambda_{m})(\dfrac{|du|^p}{p}+\dfrac{1}{4\epsilon^n}(1-|u|^2)^2)dv_g.\label{1015}
\end{align}
By condition ($P_1$), we obtain
\begin{align}
R\frac{d}{dR}\int_{ B_R(x_0)}(\dfrac{|du|^p}{p}+\dfrac{1}{4\epsilon^n}(1-|u|^2)^2)dv_g\geq
\sigma\int_{ B_R(x_0)}(\dfrac{|du|^p}{p}+\dfrac{1}{4\epsilon^n}(1-|u|^2)^2)dv_g. \nonumber
\end{align}
Integrating the above formula on $[\rho_1,\rho_2]$, finally, we can get
\begin{align}
\dfrac{\int_{B_{\rho_1}(x_0)}\dfrac{|du|^p}{p}+\dfrac{1}{4\epsilon^n}(1-|u|^2)^2 dv_g }{\rho_1^\sigma}\leq
\dfrac{\int_{B_{\rho_2}(x_0)}\dfrac{|du|^p}{p}+\dfrac{1}{4\epsilon^n}(1-|u|^2)^2 dv_g }{\rho_2^\sigma}. \nonumber
\end{align}
\end{proof}

Next we list some vanishing results which are immediate applications of the monotonicity formulae.
\begin{theorem}
Under the same condition of Theorem \ref{1014} and
\begin{align}
\int_{B_r(x_0)}\frac{|du|^p}{p}+\frac{1}{4\epsilon^n}(1-|u|^2)^2dv_g=o(r^{\sigma}), \nonumber
\end{align}
then $du=0$, that is, $u$ is a constant.
\end{theorem}
\begin{proof}
By Theorem \ref{1014}, for any $0<\rho<r$, we have the following inequality
\begin{align}
\frac{1}{\rho^{\sigma}}\int_{B_{\rho}(x_0)}\frac{|du|^p}{p}+\frac{1}{4\epsilon^n}(1-|u|^2)^2dv_g\leq\frac{1}{r^{\sigma}}\int_{B_{r}(x_0)}\frac{|du|^p}{p}+\frac{1}{4\epsilon^n}(1-|u|^2)^2dv_g. \nonumber
\end{align}
Letting $r\rightarrow +\infty$, under the the assumption, we may conclude
\begin{align}
\int_{B_{\rho}(x_0)}\frac{|du|^p}{p}+\frac{1}{4\epsilon^n}(1-|u|^2)^2dv_g=0. \nonumber
\end{align}
Since $\rho$ is arbitrary, then $du=0$.
\end{proof}

Next, we will introduce some comparison theorems in Riemannian geometry.
\begin{lemma}(cf.\cite{DW,GW,prs})\label{1028}
Let ($M,g$) be a complete Riemannian manifold with a pole $x_0$ and let $r$ be the distance function relative to $x_0$. Denote by $K_r$ the radial curvature of $M$.
\begin{enumerate}[(i)]
\item If $-\alpha^2\leq K_r\leq-\beta^2$ with $\alpha>0$, $\beta>0$, then
\begin{align}
\beta coth(\beta r)[g-dr \otimes dr]\leq Hess_g(r)\leq\alpha coth(\alpha r)[g-dr\otimes dr]. \nonumber
\end{align}
\item If $-\frac{A}{(1+r^2)^{1+\varepsilon}}\leq K_r\leq \frac{B}{(1+r^2)^{1+\varepsilon}}$ with $\varepsilon>0$, $A\geq 0$, $0\leq B <2\varepsilon$, then
\begin{align}
\frac{1-\frac{B}{2\varepsilon}}{r}\leq Hess_g(r)\leq\frac{e^{\frac{A}{2\varepsilon}}}{r}[g-dr\otimes dr]. \nonumber
\end{align}
\item If $-\frac{a^2}{1+r^2}\leq K_r\leq\frac{b^2}{1+r^2}$ with $a\geq0$, $b^2\in[0,\frac{1}{4}]$, then
\begin{align}
\frac{1+\sqrt{1-4b^2}}{2r}[g-dr\otimes dr]\leq Hess_g(r)\leq \frac{1+\sqrt{1+4a^2}}{2r}[g-dr\otimes dr]. \nonumber
\end{align}
\end{enumerate}
\end{lemma}
\begin{proof}
The case $(i)$ is standard (cf. \cite{GW}). The case $(ii)$ is discussed in \cite{DW}. For $(iii)$, see \cite{GW,Kas,prs} for details.
\end{proof}

\begin{lemma}\label{1027}
Let $(M,g)$ be a complete Riemannian manifold with a pole $x_0$ and let $r$ be the distance function relative to $x_0$. Assume that there exist two positive functions $h_1(r)$ and $h_2(r)$ such that
\begin{align}
h_1(r)[g-dr \otimes dr]\leq Hess_g(r)\leq h_2(r)[g-dr\otimes dr]\ \ \ \ \  and\ \ \ \ \  rh_2(r)\geq 1, \nonumber
\end{align}
then
\begin{align}
\sum_{i=1}^{m}\lambda_i-p\lambda_m\geq2\{1+(m-1)rh_1(r)-prh_2(r)\}.
\end{align}
\end{lemma}
Combing Lemma \ref{1028} and Lemma \ref{1027}, we can obtain the following.
\begin{lemma} \label{1029}
Let ($M,g$) be a complete Riemannian manifold with a pole $x_0$ and let $r$ be the distance function relative to $x_0$. Denote by $K_r$ the radial curvature of $M$.
\begin{enumerate}[(i)]
\item If $-\alpha^2\leq K_r\leq-\beta^2$ with $\alpha>0$, $\beta>0$ and $(m-1)\beta-p\alpha >0$, then
\begin{align}
\sum_{i}^{m}\lambda_i-p\lambda_{m}\geq 2(m-p\frac{\alpha}{\beta}). \nonumber
\end{align}
\item If $-\frac{A}{(1+r^2)^{1+\varepsilon}}\leq K_r\leq \frac{B}{(1+r^2)^{1+\varepsilon}}$ with $\varepsilon>0$, $A\geq 0$, $0\leq B <2\varepsilon$ and  $1+(m-1)(1-\frac{B}{2\varepsilon})-pe^{\frac{A}{2\varepsilon}}>0$, then
\begin{align}
\sum_{i}^{m}\lambda_i-p\lambda_{m}\geq 2[1+(m-1)(1-\frac{B}{2\varepsilon})-e^{\frac{A}{2\varepsilon}}p]. \nonumber
\end{align}
\item If $-\frac{a^2}{1+r^2}\leq K_r\leq\frac{b^2}{1+r^2}$ with $a\geq0$, $b^2\in[0,\frac{1}{4}]$ and $2+(m-1)(1+\sqrt{1-4b^2})-p(1+\sqrt{1+4a^2})>0$, then
\begin{align}
\sum_{i}^{m}\lambda_i-p\lambda_{m}\geq 2[1-\frac{p}{2}+(m-1)\frac{1+\sqrt{1-4b^2}}{2}-\frac{p}{2}\sqrt{1+4a^2}]. \nonumber
\end{align}
\end{enumerate}
\end{lemma}

\begin{corollary}\label{1013}
Let $u:(M,g)\rightarrow (R^n,h)$ be a critical point of $p$-Ginzburg-Landau energy functional from a Riemannian manifold with a pole $x_0$ to a standard Euclidean space. Assume that the radial curvature $K_r$ of M satisfies one of the following three conditions:
\begin{enumerate}[(i)]
\item $-\alpha^2\leq K_r\leq-\beta^2$ with $\alpha>0$, $\beta>0$ and $(m-1)\beta-p\alpha>0$;
\item $-\frac{A}{(1+r^2)^{1+\varepsilon}}\leq K_r\leq \frac{B}{(1+r^2)^{1+\varepsilon}}>0$ with $\varepsilon$, $A\geq 0$, $0\leq B <2\varepsilon$ and $1+(m-1)(1-\frac{B}{2\varepsilon})-pe^{\frac{A}{2\varepsilon}}>0$;
\item $-\frac{a^2}{1+r^2}\leq K_r\leq\frac{b^2}{1+r^2}$ with $a\geq0$, $b^2\in[0,\frac{1}{4}]$ and  $2+(m-1)(1+\sqrt{1-4b^2})-p(1+\sqrt{1+4a^2})>0$.
\end{enumerate}
Then
\begin{align}
\dfrac{\int_{B_{\rho_1}(x_0)}\dfrac{|du|^p}{p}+\dfrac{1}{4\epsilon^n}(1-|u|^2)^2 dv_g }{\rho_1^\sigma}\leq
\dfrac{\int_{B_{\rho_2}(x_0)}\dfrac{|du|^p}{p}+\dfrac{1}{4\epsilon^n}(1-|u|^2)^2 dv_g }{\rho_2^\sigma}. \nonumber
\end{align}
for any $0<\rho_1\leq\rho_2$, where
\begin{numcases}{\sigma=}
(m-p\frac{\alpha}{\beta});  &for $K_r$ satisfies (i) \nonumber \\
1+(m-1)(1-\frac{B}{2\varepsilon})-pe^{\frac{A}{2\varepsilon}}; & for $K_r$ satisfies (ii)\nonumber \\
\frac{2-p+(m-1)(1+\sqrt{1-4b^2})-p\sqrt{1+4a^2}}{2}. &for $K_r$ satisfies (iii) \nonumber
\end{numcases}
\end{corollary}

Corollarty \ref{1013} yields immediately the following vanishing result.
\begin{theorem}
Suppose that $u:(M,g)\rightarrow (R^n,h)$ is a critical point of $p$-Ginzburg-Landau energy funcional. Let $r$ be the distance function relative to $x_0$. If the radial curvature $K_r$ of $M$ satisfies one of the three conditions in Corollary \ref{1013} and
\begin{align}
\int_{B_r(x_0)}\frac{|du|^p}{p}+\frac{1}{4\epsilon^n}(1-|u|^2)^2dv_g=o(r^{\sigma}), \nonumber
\end{align}
where $\sigma$ is given by Corollary \ref{1013}. Then $du=0$, that is, $u$ is constant.
\end{theorem}

\begin{definition}
$E_{p}^{GL}(u)$ is said to have slowly divergent energy, if there exists a positive continuous function $\psi(r)$ such that
\begin{align}
\int_{R_1}^{+\infty}\frac{dr}{r\psi(r)}=+\infty  \nonumber
\end{align}
for some $R_1>0$, and
\begin{align}
\mathop{\lim}_{R\rightarrow \infty}\int_{B_R(x_0)}\dfrac{\dfrac{|du|^p}{p}+\dfrac{1}{4\epsilon^n}(1-|u|^2)^2}{\psi(r(x))}dv_g<\infty. \label{1025}
\end{align}
\end{definition}
\begin{theorem}
Let $u:(M,g)\rightarrow (R^n,h)$ be the critical point of $p$-Ginzburg-Landau energy functional. If $r(x)$ satisfies the condition ($P_1$) and $E_{p}^{GL}(u)$ has slowly divergent energy, then $u$ is a constant map and $u(M)\subseteq S^{n-1}$.
\end{theorem}
\begin{proof}
From Theorem \ref{1014}, we obtain
\begin{align}
&R\frac{d}{dR}\int_{ B_R(x_0)}(\dfrac{|du|^p}{p}+\dfrac{1}{4\epsilon^n}(1-|u|^2)^2)dv_g \nonumber \\
&\geq\int_{ B_R(x_0)}\frac{1}{2}(\sum_{i}^{m}\lambda_i-p\lambda_{max})(\dfrac{|du|^p}{p}+\dfrac{1}{4\epsilon^n}(1-|u|^2)^2)dv_g.\nonumber
\end{align}
If $u$ is not a constant map contained in $S^{n-1}$, there exists constants $R_0>0$ and $C_0>0$ such that
\begin{align}
\int_{ B_R(x_0)}\dfrac{|du|^p}{p}+\dfrac{1}{4\epsilon^n}(1-|u|^2)^2 dv_g\geq C_0 \nonumber
\end{align}
for any $R\geq R_0$. Thus
\begin{align}
\int_{\partial B_R(x_0)}\dfrac{|du|^p}{p}+\dfrac{1}{4\epsilon^n}(1-|u|^2)^2 ds_g\geq \frac{\sigma C_0}{R},\quad \forall R\geq R_0. \nonumber
\end{align}
Since $E_{p}^{GL}(u)$ has slowly divergent energy, then
\begin{align*}
\mathop{\lim}_{R\rightarrow \infty}\int_{B_R(x_0)}\dfrac{\dfrac{|du|^p}{p}+\dfrac{1}{4\epsilon^n}(1-|u|^2)^2}{\psi(r(x))}dv_g
&=\int_0^{\infty}\frac{dR}{\psi(R)}\int_{\partial B_R(x_0)}\dfrac{|du|^p}{p}+\dfrac{1}{4\epsilon^n}(1-|u|^2)^2 ds_g\\
\geq \int_{R_0}^{\infty}\frac{dR}{\psi(R)} &\int_{\partial B_R(x_0)}\dfrac{|du|^p}{p}+\dfrac{1}{4\epsilon^n}(1-|u|^2)^2 ds_g\\
\geq \sigma C_0&\int_{R_0}^{\infty}\frac{dr}{r\psi(r)}=\infty.
\end{align*}
It is in contradiction to (\ref{1025}).
\end{proof}

\section{Liouville theorem under the asymptotic conditions}

\ \ \ \ \ In \cite{Jin}, Jin established several Liouville theorems for harmonic maps between some Riemanian manifold under some asymptotic condition of the maps at infinity. In particular, he proved that for any harmonic map $f:(R^m,g_0)\rightarrow (N^n,h)$ ($m\geq 3$), if $f(x)\rightarrow P_0\in N^n$ as $|x|\rightarrow +\infty$, then $f$ must be a constant map. In this section, using a similar technique or idea,  we can derive a Liouville theorem for the critical points of the $p$-Ginzburg-Landau energy. To generalize this result to our case it is necessary to give more strictly asymptotic condition at infinity. We begin with evaluating the lower bounder of the energy.
\begin{proposition}\label{1024}
Assume that $u:(M^m,g)\rightarrow (R^n,h)$ is a critical point of the $p$-Ginzburg-Landau energy functional from a Riemannian manifold with a pole $x_0$ to a standard Euclidean space. $r(x)$ is the distance function relative to the pole $x_0$. If it satisfies the condition ($P_1$)
and $u(M)$ is not contained in $S^{n-1}$, then
\begin{align}
\int_{B_{R}(x_0)}\frac{|du|^p}{p}+\dfrac{1}{4\epsilon^n}(1-|u|^2)^2 dv_g \geq C(u)R^{\sigma}, \ \ \ \   R\rightarrow\infty \nonumber
\end{align}
where $C(u)$ is a positive constant only depending on $u$.
\end{proposition}
\begin{proof}
Since $u$ satisfies the conditions in Theorem \ref{1014}, we obtain
\begin{align}
\frac{\int_{B_{\rho}(x_0)}\frac{|du|^p}{p}+\dfrac{1}{4\epsilon^n}(1-|u|^2)^2 dv_g }{\rho^\sigma}\leq
\frac{\int_{B_{R}(x_0)}\frac{|du|^p}{p}+\dfrac{1}{4\epsilon^n}(1-|u|^2)^2 dv_g }{R^\sigma} \nonumber
\end{align}
for any $0<\rho<R$. Note that $u(M)$ is not contained in $S^{n-1}$, there exist some $\rho>0$ such that
\begin{align}
\int_{B_{\rho}(x_0)}\frac{|du|^p}{p}+\dfrac{1}{4\epsilon^n}(1-|u|^2)^2 dv_g >0. \nonumber
\end{align}
Denoted by $C(u)=\frac{\int_{B_{\rho}(x_0)}\frac{|du|^p}{p}+\dfrac{1}{4\epsilon^n}(1-|u|^2)^2 dv_g }{\rho^\sigma}$, then
\begin{align}
\int_{B_{R}(x_0)}\frac{|du|^p}{p}+\dfrac{1}{4\epsilon^n}(1-|u|^2)^2 dv_g\geq C(u)R^{\sigma}.\nonumber
\end{align}
\end{proof}

By Corollary \ref{1013}, and using the above proposition , we easily obtain the following corollary.
\begin{corollary}\label{1031}
Let $u:(M,g)\rightarrow (R^n,h)$ be a critical map of $p$-Ginzburg-Landau energy functional from a Riemannian manifold with a pole $x_0$ to a standard Euclidean space. Assume that the radial curvature $K_r$ of M satisfies one of the following three conditions:
\begin{enumerate}[(i)]
\item $-\alpha^2\leq K_r\leq-\beta^2$ with $\alpha>0$, $\beta>0$ and $(m-1)\beta-p\alpha>0$;
\item $-\frac{A}{(1+r^2)^{1+\varepsilon}}\leq K_r\leq \frac{B}{(1+r^2)^{1+\varepsilon}}>0$ with $\varepsilon$, $A\geq 0$, $0\leq B <2\varepsilon$ and $1+(m-1)(1-\frac{B}{2\varepsilon})-pe^{\frac{A}{2\varepsilon}}>0$;
\item $-\frac{a^2}{1+r^2}\leq K_r\leq\frac{b^2}{1+r^2}$ with $a\geq0$, $b^2\in[0,\frac{1}{4}]$ and  $2+(m-1)(1+\sqrt{1-4b^2})-p(1+\sqrt{1+4a^2})>0$.
\end{enumerate}
If $u(M)$ is not contained in $S^{n-1}$, then
\begin{align}
\int_{B_{R}(x_0)}\frac{|du|^p}{p}+\dfrac{1}{4\epsilon^n}(1-|u|^2)^2 dv_g \geq C(u)R^{\sigma}, \ \ \ \  as \ R\rightarrow\infty \nonumber
\end{align}
where $C(u)$ is a positive constant only depending on $u$, and
\begin{numcases}{\sigma=}
(m-p\frac{\alpha}{\beta});  &for $K_r$ satisfies (i) \nonumber \\
1+(m-1)(1-\frac{B}{2\varepsilon})-pe^{\frac{A}{2\varepsilon}}; & for $K_r$ satisfies (ii)\nonumber \\
\frac{2-p+(m-1)(1+\sqrt{1-4b^2})-p\sqrt{1+4a^2}}{2}.  &for $K_r$ satisfies (iii) \nonumber
\end{numcases}
\end{corollary}
\indent Next, we will show that if $u$ is the critical map of the $2$-Ginzburg-Landau energy functional and it is uniformly bounded, the condition ($P_1$) may be replaced by\\
\indent ($\widetilde{P_1}$): The left hand side of the inequality in ($P_1$) is nonnegative on the whole $M$ and there exists an $R_0>0$ such that ($P_1$) holds for $r(x)\geq R_0$.\\
\indent To prove this assertion, we start with the following lemmas.
\begin{lemma}\label{1035}
Suppose that $u:(M^m,g)\rightarrow (R^n,h)$ is a critical map of the 2-Ginzburg-Landau energy functional and $u$ is uniformly bounded. If $u$ is constant in an open set of $M$, then $u$ is constant on $M$.
\end{lemma}
\begin{proof}
From the Euler-Largrange equation, we have
\begin{align}
\Delta u+\frac{1}{\epsilon^n}(1-|u|^2)u=0. \nonumber
\end{align}
Since $u$ is bounded, using unique continuation theorem in \cite{A}, one can deduce that $u$ is constant on $M$.
\end{proof}

\begin{lemma}\label{1018}
Assume that $u$ satisfies ($\widetilde{P}_1$). If $du$ is not identically zero, then $E^{GL}_{2}(u)=+\infty$.
\end{lemma}
\begin{proof}
By co-area formula, we have
\begin{align}
\int_M \frac{|du|^2}{2}+\frac{1}{4\epsilon^n}(1-|u|^2)^2 dv_g
&=\int_0^{+\infty}\dfrac{dr}{r}\cdot r\int_{\partial B_r(x_0)}(\dfrac{|du|^2}{2}+\dfrac{1}{4\epsilon^n}(1-|u|^2)^2)\dfrac{1}{|\nabla r|}ds_g\nonumber\\
&=\int_0^{+\infty}\frac{dr}{r}\cdot r\int_{\partial B_r(x_0)}\dfrac{|du|^2}{2}+\dfrac{1}{4\epsilon^n}(1-|u|^2)^2ds_g. \nonumber
\end{align}
If $E_{2}^{GL}(u)<+\infty$, we can choose a sequence $\{r_i\}$ such that
\begin{align}
\mathop{\lim}_{r_i\rightarrow +\infty} r_i\int_{\partial B_{r_i}(x_0)}\dfrac{|du|^2}{2}+\dfrac{1}{4\epsilon^n}(1-|u|^2)^2ds_g=0. \label{1017}
\end{align}
Since $r(x)$ satisfies ($\widetilde{P_1}$), by (\ref{1015}) we obtain
\begin{align}
r\frac{d}{dr}\int_{ B_r(x_0)\backslash B_{R_0}(x_0)}\dfrac{|du|^2}{2}+\dfrac{1}{4\epsilon^n}(1-|u|^2)^2dv_g\geq
\sigma\int_{ B_r(x_0)\backslash B_{R_0}(x_0) }\dfrac{|du|^2}{2}+\dfrac{1}{4\epsilon^n}(1-|u|^2)^2dv_g, \nonumber
\end{align}
that is
\begin{align}
r\int_{ \partial B_r(x_0)}\frac{|du|^2}{2}+\frac{1}{4\epsilon^n}(1-|u|^2)^2dv_g\geq
\sigma\int_{ B_r(x_0)\backslash B_{R_0}(x_0) }\dfrac{|du|^2}{2}+\dfrac{1}{4\epsilon^n}(1-|u|^2)^2dv_g. \nonumber
\end{align}
Let $r=r_i$ tend to infinity in above inequality, it follows from (\ref{1017}) that
\begin{align}
\int_{M\setminus B_{R_0}(x_0)}\frac{|du|^2}{2}+\frac{1}{4\epsilon^n}(1-|u|^2)^2dv_g=0. \nonumber
\end{align}
It follows from Lemma \ref{1035} that $du=0$ on $M$ which contradicts with the condition. Therefore $E_{2}^{GL}(u)=+\infty$.
\end{proof}

\begin{proposition}
Assume that $u:(M^m,g)\rightarrow (R^n,h)$ is a critical point of the $2$-Ginzburg-Landau energy functional and $u$ is uniformly bounded . If $r(x)$ satisfies $\widetilde{P_1}$
and $u(M)$ is not contained in $S^{n-1}$, then
\begin{align}
\int_{B_{R}(x_0)}\frac{|du|^2}{2}+\dfrac{1}{4\epsilon^n}(1-|u|^2)^2 dv_g \geq C(u)R^{\sigma}, \nonumber
\end{align}
where $R$ is sufficiently large and $C(u)$ is a positive constant only depending on $u$.
\end{proposition}
\begin{proof}
Taking $D=B_R(x_0)\backslash B_{R_0}(x_0)$ and $X=r\frac{\partial}{\partial r}=\frac{1}{2}\nabla r^2$, by (\ref{1006}) we get
\begin{align}
\int_{\partial B_R(x_0)}S_{2,u}^{GL}(X,\nu)ds_g-\int_{\partial B_{ R_0}(x_0)}S_{2,u}^{GL}(X,\nu)ds_g
&=\int_{B_{R}(x_0)\backslash B_{R_0}(x_0)} <S_{2,u}^{GL},\nabla \theta_X>dv_g\nonumber\\
\geq \int_{B_{R}(x_0)\backslash B_{R_0}(x_0)}[\frac{1}{2}(\sum_i^m \lambda_i-2\lambda_{max})]&(\dfrac{|du|^2}{2}+\dfrac{1}{4\epsilon^n}(1-|u|^2)^2)dv_g. \nonumber
\end{align}
Denoting $\int_{\partial B_{R_0}(x_0)}S_{p=2,u}^H(X,\nu)ds_g$ by $H(R_0)$, then (\ref{1011}) and condition ($\widetilde{P}_1$) yield
\begin{align}
R\int_{\partial B_R(x_0)}\frac{|du|^2}{2}+\frac{1}{4\epsilon^n}(1-|u|^2)^2ds_g-2H(R_0)\geq\sigma\int_{B_R(x_0)\backslash B_{R_0}(x_0)}\frac{|du|^2}{2}+\frac{1}{4\epsilon^n}(1-|u|^2)^2dv_g. \nonumber
\end{align}
It also can be written as
\begin{align}
R\frac{d}{dR}\{\int_{B_R(x_0)\backslash B_{R_0}(x_0)}\frac{|du|^2}{2}+\frac{1}{4\epsilon^n}(1-|u|^2)^2dv_g+\frac{2H(R_0)}{\sigma}\}\nonumber \\
\geq\sigma\{\int_{B_R(x_0)\backslash B_{R_0}(x_0)}\frac{|du|^2}{2}+\frac{1}{4\epsilon^n}(1-|u|^2)^2dv_g+\frac{2H(R_0)}{\sigma}\}. \nonumber
\end{align}
From Lemma \ref{1018}, we know that $E_{2}^{GL}(u)=+\infty$. Therefore, when $R$ is sufficiently large, we get
\begin{align}
\int_{B_R(x_0)\backslash B_{R_0}(x_0)}\frac{|du|^2}{2}+\frac{1}{4\epsilon^n}(1-|u|^2)^2dv_g+\frac{2H(R_0)}{\sigma}>0. \nonumber
\end{align}
Then
\begin{align}
\frac{\frac{d}{dR}\{\int_{B_R(x_0)\backslash B_{R_0}(x_0)}\frac{|du|^2}{2}+\frac{1}{4\epsilon^n}(1-|u|^2)^2dv_g+\frac{2H(R_0)}{\sigma}\}}{\int_{B_R(x_0)\backslash B_{R_0}(x_0)}\frac{|du|^2}{2}+\frac{1}{4\epsilon^n}(1-|u|^2)^2dv_g+\frac{2H(R_0)}{\sigma}}\geq \frac{\sigma}{R}.\nonumber
\end{align}
Fixing some $R_0<\overline{R}<R$ and integrating the above formula on [$\overline{R},R$], we get
\begin{align}
\int_{B_R(x_0)}\frac{|du|^2}{2}+\frac{1}{4\epsilon^n}(1-|u|^2)^2dv_g\geq\{H(R_0,\overline{R})-\frac{2H(R_0)}{\sigma R^{\sigma}}\}R^{\sigma},\nonumber
\end{align}
where $H(R_0,\overline{R})=\frac{\int_{B_{\overline{R}}(x_0)\backslash B_{R_0}(x_0)}\frac{|du|^2}{2}+\frac{1}{4\epsilon^n}(1-|u|^2)^2dv_g+\frac{2H(R_0)}{\sigma}}{\overline{R}^{\sigma}}$.
When $R\rightarrow +\infty$, $\frac{2H(R_0)}{\sigma R^{\sigma}}$ can be controlled by $H(R_0,\overline{R})$. Consequently
\begin{align}
\int_{B_R(x_0)}\frac{|du|^2}{2}+\frac{1}{4\epsilon^n}(1-|u|^2)^2dv_g\geq C(u)R^{\sigma},\nonumber
\end{align}
where $C(u)$ is a constant depending on the map $u$.
\end{proof}

Next we will use the assumption for the map at infinity to derivative an upper bound for the growth rate. The condition that we will assume for $u$ is as follow:\\
\indent $(P_2)$ There exists a positive constant $\widetilde{\sigma}$ less than $\sigma$ in $(P_1)$ such that
\begin{align}
\ \ \ \ \ \ \ max_{r(x)=r}h^2(u(x),P_0)\leq r^{\widetilde{\sigma}}\int_r^{+\infty}\frac{ds}{vol(\partial B_s(x_0))}  \ \ \ \ \  \ for\ \ \  r(x)\gg 1. \nonumber
\end{align}
\begin{theorem} \label{1026}
Let $u:(M,g)\rightarrow(R^n,h)$ be a critical point of p-Ginzburg-Landau energy functional. Suppose that $|du|^{p-2}$ is uniformly bounded and $r(x)$ satisfies the condition $(P_1)$. If $u(x)\rightarrow P_0 \in S^{n-1}$ and $u$ satisfies the  condition $(P_2)$, the $u$ must be a constant map.
\end{theorem}
\begin{proof}
Suppose the critical point $u$ is not constant, then by Proposition \ref{1024}, the energy of $u$ must be infinite. That is, $\int_{B_R(x_0)}\frac{|du|^p}{p}+\frac{1}{4\epsilon^n}(1-|u|^2)^2dv_g\rightarrow +\infty$ as $r(x)\rightarrow +\infty$. \\
\indent Since $P_0=(c_1,c_2,\cdots,c_{\alpha},\cdots,c_n)\in S^{n-1}$, then $\sum_{\alpha=1}^n c_{\alpha}^2=1$. it is clear that we can choose a  orthogonal matrix $A$ such that $AP_0=\widetilde{P}_0=(\widetilde{c}_1,\widetilde{c}_2,\cdots,\widetilde{c}_{\alpha},\cdots,\widetilde{c}_n),\widetilde{c}_{\alpha}\neq 0$, for each
$ \alpha=1,2,\cdots n$. Clearly if $u$ is the critical point of $p$-Ginzburg-Landau energy functional, then $Au$ is also the critical point. Hence without loss of generality we may assume that $u(x)\rightarrow p_0\in S^{n-1}$, where $p_0=(c_1,c_2,\cdots,c_{\alpha},\cdots,c_n),c_{\alpha}\neq 0,\alpha=1,2,\cdots n.$\\
\indent Now the assumption that $u(x)\rightarrow P_0$ as $r(x)\rightarrow +\infty$ implies that there exists an $R_1>0$ and  a neighbourhood $U$ of $P_0$ such that for $r(x)> R_1$, $u(x)\in U$ and $u_{\alpha}\neq 0$ for any $\alpha=1,2,\cdots,n$. \\
\indent For $\omega \in C_0^2(M\backslash B_{R_1}(x_0),U)$, we consider the variation $u+t\omega:M\rightarrow R^n$ defined as follows:
\[(u+t\omega)(q)=\begin{cases}
  u(q)        \quad   & q\in B_{R_1}(x_0),\\
  (u+t\omega)(q) &q\in M^m\backslash B_{R_1}(x_0)
\end{cases}\]
for sufficiently small $t$. Since $u$ is the critical point of $p$-Ginzburg-Landau energy functional, we have
\[\frac{d}{dt}|_{t=0} E_{p}^{GL}(u+t\omega)=0\]
that is,
\begin{align}
\int_{M^m\backslash B_{R_1}(x_0)}
|du|^{p-2}\sum_{k=1}^{n}g^{ij}\frac{\partial u_k}{\partial x_i}\frac{\partial \omega_k}{\partial x_j}-\dfrac{1}{\epsilon^n}(1-\sum_{k=1}^{n}u_k^2)\sum_{k=1}^{n}u_k\omega_k dv_g=0. \label{1019}
\end{align}
Choose $\omega(x)=\phi(r(x))\widetilde{u}(x)$ in (\ref{1019}) for $\phi(t)\in C_0^\infty(R_1,\infty),\widetilde{u}_k=\frac{u_k^2-c_k^2}{u_k}$, we obtain
\begin{align}
\int_{M^m \setminus B_{R_1}(x_0)}|du|^{p-2}&\sum_{k=1}^ng^{ij}\frac{\partial u_k}{\partial x_i}\frac{\partial \tilde{u_k}}{\partial x_j}\phi(r(x))-\frac{1}{\epsilon^n}(1-\sum_{k=1}^nu_k^2)\phi(r(x))\sum_{k=1}^nu_k\tilde{u_k}dv_g \nonumber \\
&=\int_{M^m \setminus B_{R_1}(x_0)}|du|^{p-2}\sum_{k=1}^ng^{ij}\frac{\partial u_k}{\partial x_i}\tilde{u_k}\frac{\partial \phi(r(x))}{\partial x_j}dv_g. \label{1020}
\end{align}
By a standard approximation argument, (\ref{1020}) holds for Lipschitz functions $\phi$ with compact support. \\
\indent For $0<\theta\leq 1$, define
\[\varphi_\theta(t)=
\begin{cases}
1 \quad & t\leq 1;\\
1+\frac{1-t}{\theta} &1<t<1+\theta;\\
0 & t\geq 1+\theta.
\end{cases}\]
In (\ref{1020}), choose the Lipschitz function $\phi(r(x))$ to be
\[\phi(r(x))=\varphi_\theta(\frac{r(x)}{R})(1-\varphi_1(\frac{r(x)}{R_1})),\ \ R>2R_1\ \ \  and \ \ \ R_2=2R_1.\]
Then the first term on left hand side of (\ref{1020}) becomes
\begin{align}
&\int_{M^m \setminus B_{R_1}(x_0)}|du|^{p-2}\sum_{k=1}^ng^{ij}\frac{\partial u_k}{\partial x_i}\frac{\partial \tilde{u_k}}{\partial x_j}\phi(r(x))dv_g\nonumber \\
&=\int_{B_{R_2}(x_0) \setminus B_{R_1}(x_0)}|du|^{p-2}\sum_{k=1}^ng^{ij}\frac{\partial u_k}{\partial x_i}\frac{\partial \tilde{u_k}}{\partial x_j}[1-\varphi_1(\frac{r(x)}{R_1})]dv_g\nonumber \\
&+\int_{B_{R}(x_0) \setminus B_{R_2}(x_0)}|du|^{p-2}\sum_{k=1}^ng^{ij}\frac{\partial u_k}{\partial x_i}\frac{\partial \tilde{u_k}}{\partial x_j}dv_g\nonumber \\
&+\int_{B_{R(1+\theta)}(x_0) \setminus B_{R}(x_0)}|du|^{p-2}\sum_{k=1}^ng^{ij}\frac{\partial u_k}{\partial x_i}\frac{\partial \tilde{u_k}}{\partial x_j}\varphi_{\theta}(\frac{r(x)}{R})dv_g. \label{1021}
\end{align}
The second term on left hand side of (\ref{1020}) becomes
\begin{align}
&-\int_{M^m\setminus B_{R_1(x_0)}}\frac{1}{\epsilon^n}(1-\sum_{k=1}^nu_k^2)\phi(r(x))\sum_{k=1}^nu_k\tilde{u_k}dv_g \nonumber \\
&=-\int_{B_{R_2}(x_0)\setminus B_{R_1(x_0)}}\frac{1}{\epsilon^n}(1-\sum_{k=1}^nu_k^2)[1-\varphi_1(\frac{r(x)}{R_1})]\sum_{k=1}^nu_k\tilde{u_k}dv_g \nonumber \\
&-\int_{B_R(x_0) \setminus B_{R_2(x_0)}}\frac{1}{\epsilon^n}(1-\sum_{k=1}^nu_k^2)\sum_{k=1}^nu_k\tilde{u_k}dv_g \nonumber \\
&-\int_{B_{R(1+\theta)}(x_0) \setminus B_{R(x_0)}}\frac{1}{\epsilon^n}(1-\sum_{k=1}^nu_k^2)\varphi_{\theta}(\frac{r(x)}{R})\sum_{k=1}^nu_k\tilde{u_k}dv_g. \label{1022}
\end{align}
We can also compute the right hand side of (\ref{1020}) as follows
\begin{align}
&\int_{M^m \setminus B_{R_1}(x_0)}|du|^{p-2}\sum_{k=1}^ng^{ij}\frac{\partial u_k}{\partial x_i}\tilde{u_k}\frac{\partial \phi(r(x))}{\partial x_j}dv_g\nonumber \\
&=\int_{B_{R_2}(x_0) \setminus B_{R_1}(x_0)}|du|^{p-2}\sum_{k=1}^ng^{ij}\frac{\partial u_k}{\partial x_i}\tilde{u_k}\frac{\partial \varphi_1(\frac{r(x)}{R_1})}{\partial x_j}dv_g\nonumber \\
&+\frac{1}{R\theta}\int_{B_{R(1+\theta)}(x_0) \setminus B_{R}(x_0)}|du|^{p-2}\sum_{k=1}^ng^{ij}\frac{\partial u_k}{\partial x_i}\tilde{u_k}\frac{\partial r(x)}{\partial x_j}dv_g. \label{1023}
\end{align}
Set
\begin{align}
&D(R_1)=\int_{B_{R_2}(x_0) \setminus B_{R_1}(x_0)}|du|^{p-2}\sum_{k=1}^ng^{ij}\frac{\partial u_k}{\partial x_i}\frac{\partial \tilde{u_k}}{\partial x_j}[1-\varphi_1(\frac{r(x)}{R_1})]dv_g \nonumber \\
&\ \ \ \ -\int_{B_{R_2}(x_0)\setminus B_{R_1(x_0)}}\frac{1}{\epsilon^n}(1-\sum_{k=1}^nu_k^2)[1-\varphi_1(\frac{r(x)}{R_1})]\sum_{k=1}^nu_k\tilde{u_k}
dv_g \nonumber \\
&\ \ \ \ -\int_{B_{R_2}(x_0) \setminus B_{R_1}(x_0)}|du|^{p-2}\sum_{k=1}^ng^{ij}\frac{\partial u_k}{\partial x_i}\tilde{u_k}\frac{\partial \varphi_1(\frac{r(x)}{R_1})}{\partial x_j}dv_g. \nonumber
\end{align}
Substitute (\ref{1021}), (\ref{1022}) and (\ref{1023}) into (\ref{1020}), then letting $\theta \rightarrow 0$, we have
\begin{align}
&\int_{B_{R}(x_0) \setminus B_{R_2}(x_0)}|du|^{p-2}\sum_{k=1}^ng^{ij}\frac{\partial u_k}{\partial x_i}\frac{\partial \tilde{u_k}}{\partial x_j}-\frac{1}{\epsilon^n}(1-\sum_{k=1}^nu_k^2)\sum_{k=1}^nu_k\tilde{u_k}dv_g+D(R_1)\nonumber \\
&\ \ \ \ \ \ \ \ =\int_{\partial B_{R}(x_0)}|du|^{p-2}\sum_{k=1}^n\frac{\partial u_k}{\partial x_i}\nu^i\tilde{u_k}ds_g,\label{1000}
\end{align}
where $\nu^i=\frac{\partial r}{\partial x_j}$ and $\nu=\nu^i\frac{\partial}{\partial x_i}$ is the outer normal vector field along $B_{R_0}(x_0)$. \\
\indent Note that $\widetilde{u_k}=\frac{u_k^2-c_k^2}{u_k}$. Thus
$\frac{\partial \widetilde{u_k}}{\partial x_j}=(1+\frac{c_k^2}{u_k^2})\frac{\partial u_k}{\partial x_j}$. Then (\ref{1000}) becomes
\begin{align}
&\int_{B_{R}(x_0) \setminus B_{R_2}(x_0)}|du|^{p-2}\sum_{k=1}^ng^{ij}\frac{\partial u_k}{\partial x_i}
\frac{\partial u_k}{\partial x_j}(1+\frac{c_k^2}{u_k^2})+\frac{1}{\epsilon^n}(1-|u|^2)^2dv_g+D(R_1)\nonumber \\
&\ \ \ \ \ \ \ \ \ =\int_{\partial B_{R}(x_0)}|du|^{p-2}\sum_{k=1}^n\frac{\partial u_k}{\partial x_i}\nu^i\tilde{u_k}ds_g. \label{1001}
\end{align}
Indeed,
\begin{align}
\sum_{k=1}^n\frac{\partial u_k}{\partial x_i}\nu^i\widetilde{u_k}=\sum_{k=1}^n\langle \frac{\partial u_k}{\partial x_i}dx^i\otimes\frac{\partial}{\partial y_k}, \frac{\partial r}{\partial x_j}\widetilde{u_k}dx^i\otimes\frac{\partial}{\partial y_k}\rangle. \nonumber
\end{align}
Therefore
\begin{align}
&\int_{\partial B_{R}(x_0)}|du|^{p-2}\sum_{k=1}^n\frac{\partial u_k}{\partial x_i}\nu^i\tilde{u_k}ds_g \nonumber \\
&=\int_{\partial B_{R}(x_0)}|du|^{p-2}\sum_{k=1}^n\langle \frac{\partial u_k}{\partial x_i}dx^i\otimes\frac{\partial}{\partial y_k}, \frac{\partial r}{\partial x_j}\widetilde{u_k}dx^i\otimes\frac{\partial}{\partial y_k}\rangle ds_g \nonumber \\
&\leq \int_{\partial B_R(x_0)}|du|^{p-2}|\frac{\partial u_k}{\partial x_i}dx^i\otimes \frac{\partial}{\partial y_k}||\frac{\partial r}{\partial x_j}\widetilde{u_k}dx^i\otimes \frac{\partial}{\partial y_k}|ds_g \nonumber  \\
&\leq \sqrt{\int_{\partial B_R(x_0)}|du|^{p-2}|\frac{\partial u_k}{\partial x_i}dx_i\otimes \frac{\partial}{\partial y_k}|^2ds_g}\sqrt{\int_{\partial B_R(x_0)}|du|^{p-2}|\frac{\partial r}{\partial x_j}\widetilde{u_k}dx^i\otimes \frac{\partial}{\partial y_k}|^2ds_g} \nonumber \\
&=\sqrt{\int_{\partial B_{R}(x_0)}|du|^pds_g}\sqrt{\int_{\partial B_R(x_0)}|du|^{p-2}(\sum_{k=1}^n\widetilde{u_k}^2)ds_g} \nonumber \\
&\leq\sqrt{\int_{\partial B_{R}(x_0)}|du|^p+\frac{1}{\epsilon^n}(1-|u|^2)^2ds_g}\sqrt{\int_{\partial B_R(x_0)}|du|^{p-2}(\sum_{k=1}^n\widetilde{u_k}^2)ds_g}. \label{1002}
\end{align}
Next, for any $R\geq R_2$ we let
\begin{align}
G(R)=\int_{B_R(x_0)\setminus B_{R_2}(x_0)}|du|^p+\frac{1}{\epsilon^n}(1-|u|^2)^2dv_g+D(R_1). \nonumber
\end{align}
Then
\begin{align}
G^{\prime}(R)=\int_{\partial B_{R}(x_0)}|du|^p+\frac{1}{\epsilon^n}(1-|u|^2)^2ds_g. \nonumber
\end{align}
Hence from (\ref{1001}), (\ref{1002}) and the fact that $1+\frac{c_{\alpha}^2}{u_{\alpha}^2}\geq 1$ for any $\alpha=1,2,\cdots,n$.
\begin{align}
G^2(R)\leq G^{\prime}(R)\int_{\partial B_R(x_0)}|du|^{p-2}\sum_{k=1}^n\widetilde{u_k}^2 ds_g. \nonumber
\end{align}
\indent On the other hand, we have the following estimate.
\begin{align}
G(R)-D(R_1)&=\int_{B_R(x_0)\setminus B_{R_2}(x_0)} |du|^p+\frac{1}{\epsilon^n}(1-|u|^2)^2dv_g\nonumber \\
&=p\int_{B_R(x_0)\setminus B_{R_2}(x_0)}\frac{|du|^p}{p}dv_g+4\int_{B_R(x_0)\setminus B_{R_2}(x_0)}\frac{1}{4\epsilon^n}(1-|u|^2)^2dv_g \nonumber \\
&\geq min\{p,4\}\int_{B_R(x_0)\setminus B_{R_2}(x_0)} [\frac{|du|^p}{p}+\frac{1}{4\epsilon^n}(1-|u|^2)^2]dv_g  \nonumber \\
&=C_{p,4}\int_{B_R(x_0)\setminus B_{R_2}(x_0)} [\frac{|du|^p}{p}+\frac{1}{4\epsilon^n}(1-|u|^2)^2]dv_g. \nonumber
\end{align}
\indent Since $E_{p}^{GL}(u)$ is infinity, there exists $\widetilde{R}\geq R_2$, $G(R)>0$ for any $R>\widetilde{R}$. \\
\indent Set $J(R)=\int_{\partial B_R(x_0)}|du|^{p-2}\sum_{k=1}^n\widetilde{u_k}^2 ds_g$. Then
\begin{align}
G^2(R)\leq G^{\prime}(R)J(R). \nonumber
\end{align}
For any $\widetilde{R}\leq R< R_4$, we have
\begin{align}
\int_R^{R_4}\frac{G^{\prime}(r)}{G^2(r)}dr\geq \int_R^{R_4}\frac{dr}{J(r)}, \nonumber \\
\frac{1}{G(R)}-\frac{1}{G(R_4)}\geq \int_R^{R_4}\frac{dr}{J(r)}. \nonumber
\end{align}
When $R_4\rightarrow +\infty$, we get $G(R)\leq \frac{1}{\int_R^{\infty}\frac{dr}{J(r)}}$. \\
\indent Note the fact that $|du|^{p-2}$ is uniformly bounded. Using the condition ($P_2$) and $u(x)\rightarrow P_0$ as $r(x)\rightarrow +\infty$, we get
\begin{align}
J(r)&=\int_{\partial B_r(x_0)}|du|^{p-2}\sum_{k=1}^n\widetilde{u_k}^2 ds_g \nonumber \\
    &\leq\int_{\partial B_r(x_0)}|du|^{p-2}\tau(r)ds_g \nonumber \\
    &\leq \widetilde{C}\tau(r)\cdot vol(\partial B_r(x_0)), \nonumber
\end{align}
where $\widetilde{C}$ is a constant only depending on $u$ and $\tau(r)$ is chosen in such a way that
\begin{enumerate}
  \item $\tau(r)$ is nonincreasing on $(\widetilde{R}, +\infty)$ and $\tau(r)\rightarrow 0$ as $r\rightarrow +\infty$;
  \item $\tau(r)\geq \max_{r(x)=r}\{\sum_{k=1}^n\widetilde{u_k}^2\}$;
  \item $\tau(r)\leq C_{P_0} r^{\widetilde{\sigma}} \cdot \int_r^{+\infty}\frac{ds}{vol(\partial B_{s}(x_0))}$,
\end{enumerate}
where $C_{P_0}$ is a constant only depending on $P_0$.
Then we can derive
\begin{align}
\int_R^{+\infty}\frac{dr}{J(r)}\geq \int_R^{+\infty}\frac{dr}{\widetilde{C}\tau(r)vol(\partial B_r(x_0))}\geq \frac{1}{\widetilde{C} \tau(R)}\int_R^{+\infty}\frac{dr}{vol(\partial B_r(x_0))}\geq\frac{1}{C_1 R^{\widetilde{\sigma}}}, \nonumber
\end{align}
where $C_1=\widetilde{C}\cdot C_{P_0}$. Hence $G(R)\leq C_1 R^{\widetilde{\sigma}}$ for any $R\geq \widetilde{R}$. By the definition of $G(R)$, we have
\begin{align}
&\int_{B_R(x_0)} \frac{|du|^p}{p}+\frac{1}{4\epsilon^n}(1-|u|^2)^2dv_g \nonumber \\
&\leq \frac{C_1}{C_{p,4}}R^{\widetilde{\sigma}}-\frac{D(R_1)}{C_{p,4}}+\int_{B_{R_2}(x_0)} \frac{|du|^p}{p}+\frac{1}{4\epsilon^n}(1-|u|^2)^2dv_g \nonumber \\
&=\{CR^{\widetilde{\sigma}-\sigma}+\frac{C(u)}{R^{\sigma}}\}R^{\sigma}, \nonumber
\end{align}
where $C(u)$ is constant only depending on $u$. Since $\widetilde{\sigma}<\sigma$, it contradicts with Proposition \ref{1024}.
\end{proof}

In \cite{prs}, the authors give the volume growth estimates under Ricci curvature conditions. Hence, applying the results to the following cases, the right side of the inequality in condition ($P_2$) can be expressed as a polynomial.
\begin{corollary} \label{1030}
Let $u:(M,g)\rightarrow (R^n,h)$ be a critical point of $p$-Ginzburg-Landau energy functional. Suppose that $|du|^{p-2}$ is uniformly bounded. Assume that the radial curvature $K_r$ of M satisfies the following condition
\begin{align}
-\frac{A}{(1+r^2)^{1+\varepsilon}}\leq K_r\leq \frac{B}{(1+r^2)^{1+\varepsilon}}\ \ \ \ \  with \ \ \varepsilon>0, \nonumber
\end{align}
where $A\geq 0$, $0\leq B <2\varepsilon$ and $1+(m-1)(1-\frac{B}{2\varepsilon})-pe^{\frac{A}{2\varepsilon}}>0$. If $u(x)\rightarrow P_0\in S^{n-1}$ as $r(x)\rightarrow +\infty$, and
\begin{align}
max_{r(x)=r}h^2(u(x),P_0)\leq\frac{r^{\widetilde{\sigma}-(m-2)}}{(m-2)\omega_me^{\frac{(m-1)A}{2\varepsilon}}},   \nonumber
\end{align}
then $u$ must be a constant map. Here $\omega_m$ is the $(m-1)$-volume of the unit sphere in $R^m$ and $\widetilde{\sigma}$ is any positive constant such that $\widetilde{\sigma}<[1+(m-1)(1-\frac{B}{2\varepsilon})-e^{\frac{A}{2\varepsilon}}p]$.
\end{corollary}
\begin{proof}
For the condition on the radial curvature $K_r$, it follows that
\begin{align}
Ric(x)\geq -\frac{(m-1)A}{(1+r^2(x))^{1+\varepsilon}}, \nonumber
\end{align}
for any $x\in M$. Then a direct calculation yields
\begin{align}
\int_0^{+\infty}\frac{Ar}{(1+r^2)^{1+\varepsilon}}dr=\frac{A}{2\varepsilon}. \nonumber
\end{align}
By the volume comparison theorem (cf. Corollary $2.17$ in [PRS]), we obtain
\begin{align}
vol(\partial B_{r}(x_0))\leq \omega_me^{\frac{(m-1)A}{2\varepsilon}}r^{m-1}, \nonumber
\end{align}
where $\omega_m$ is the $(m-1)$-volume of the unit sphere in $R^m$, and thus
\begin{align}
(\int_R^{+\infty}\frac{dr}{vol(\partial B_r(x_0))})^{-1}\leq (m-2)\omega_me^{\frac{(m-1)A}{2\varepsilon}}R^{m-2} \nonumber
\end{align}
for $R\gg1$. Using Corollary \ref{1031} and Theorem \ref{1026}, we can get the result.
\end{proof}
\begin{corollary}
Let $u:(M,g)\rightarrow (R^n,h)$ be a critical point of $p$-Ginzburg-Landau energy functional. Suppose that $|du|^{p-2}$ is uniformly bounded. Assume that the radial curvature $K_r$ of M satisfies the following condition
\begin{align}
-\frac{a^2}{1+r^2}\leq K_r\leq\frac{b^2}{1+r^2}\nonumber
\end{align}
with $a\geq0$, $b^2\in[0,\frac{1}{4}]$ and $2+(m-1)(1+\sqrt{1-4b^2})-p(1+\sqrt{1+4a^2})>0$. If $u(x)\rightarrow P_0\in S^{n-1}$ as $r(x)\rightarrow +\infty$ and
\begin{align}
max_{r(x)=r}h^2(u(x),P_0)\leq Cr^{\widetilde{\sigma}-(m-1)A^{\prime}+1},    \nonumber
\end{align}
then $u$ must be a constant map. Here $A^{\prime}=\frac{1+\sqrt{1+4a^2}}{2}$ and $\widetilde{\sigma}$ is any positive constant such that $\widetilde{\sigma}<[1-\frac{p}{2}+(m-1)\frac{1+\sqrt{1-4b^2}}{2}-\frac{p}{2}\sqrt{1+4a^2}]$.
\end{corollary}
\begin{proof}
\indent For the condition on the radial curvature $K_r$, it follows that
\begin{align}
Ric(x)\geq -\frac{(m-1)a^2}{1+r^2(x)} \nonumber
\end{align}
for any $x\in M$. We can also use the volume comparison theorem (cf. Corollary $2.17$ in [PRS]), then
\begin{align}
vol(\partial B_R(x_0))\leq CR^{(m-1)A^{\prime}}, \nonumber
\end{align}
where $C$ is suitable constant and $A^{\prime}=\frac{1+\sqrt{1+4a^2}}{2}$. Thus
\begin{align}
(\int_R^{+\infty}\frac{dr}{vol(\partial B_r(x_0))})^{-1}\leq CR^{(m-1)A^{\prime}-1}\nonumber
\end{align}
for $R\gg 1$. Using Corollary \ref{1031} and Theorem \ref{1026}, we can get the result.
\end{proof}

If the Riemannian manifold $(M,g)$ is the standard Euclidean space $(R^m,h)$, the eigenvalues of $Hessg(r^2)$ are all $2$. When $p=2$, $\frac{1}{2}(\sum_{i=1}^m\lambda_i-2\lambda_m)=m-2$. Thus we have the following result.
\begin{corollary}
Let $u:(R^m,h)\rightarrow(R^n,h)$ be a critical point of 2-Ginzburg-Landau energy functional. If $u(x)\rightarrow P_0 \in S^{n-1}$ as $r(x)\rightarrow +\infty$, $u$ must be a constant map contained in $S^{n-1}$.
\end{corollary}
\begin{proof}
Since $(M,g)$ is the standard Euclidean space, from the proofs in Theorem \ref{1026}, we obtain
\begin{align}
\int_R^{+\infty}\frac{dr}{J(r)}\geq \frac{C_m}{\tau(R)}\frac{1}{R^{m-2}}, \ \ \ \ \ for \ \ any \ \ R\geq \widetilde{R} \nonumber
\end{align}
where $C_m$ is a positive constant only depending on $m$ and $\tau(r)$ satisfies the following conditions
\begin{enumerate}
  \item $\tau(r)$ is nonincreasing on $(\widetilde{R}, +\infty)$ and $\tau(r)\rightarrow 0$ as $r\rightarrow +\infty$;
  \item $\tau(r)\geq \max_{r(x)=r}\{\sum_{k=1}^n\widetilde{u_k}^2\}$.
\end{enumerate}
Then
\begin{align}
\int_{B_R(x_0)} \frac{|du|^2}{2}+\frac{1}{4\epsilon^n}(1-|u|^2)^2dv_g\leq C(\tau(R)+\frac{C(u)}{R^{m-2}})R^{m-2}. \nonumber
\end{align}
\end{proof}

\section{Constant Dirichlet Boundary-value Problems}
\begin{definition}
A bounded domain $D\subseteq M$ with $C^1$ boundary $\partial D$ is called starlike if there exists an interior point $x_0\in D$ such that
\begin{align}
\langle \frac{\partial}{\partial r_{x_0}}, \nu \rangle|_{\partial D}\geq 0 \nonumber
\end{align}
where $\nu$ is the unit outer normal to $\partial D$, and the vector field $\frac{\partial}{\partial r_{x_0}}$ is the unit vector field such that for any $x\in D\setminus \{x_0\}\cup \partial D$, $\frac{\partial}{\partial r_{x_0}}$ is the unit vector tangent to the unique geodesic joining $x_0$ to $x$ and pointing away from $x_0$.
\end{definition}
\begin{theorem}
Suppose $M$ satisfies the same condition of Theorem \ref{1014} and $D\subseteq M$ is a bounded starlike domain with $C^1$ boundary. If $u:(M,g)\rightarrow (R^n,h)$ is a critical point of The $p$-Ginzburg-Landau energy functional and $u|_{\partial D}\subseteq S^{n-1}$ is constant, then $u|_D$ is constant.
\end{theorem}
\begin{proof}
Set $X=r\frac{\partial}{\partial r}$, where $r=r_{x_0}$. From the proof of Theorem \ref{1014}, we have
\begin{align}
\int_D\langle S_{u,p}^{GL},\frac{1}{2}L_Xg\rangle dv_g \geq \sigma\int_D\frac{|du|^p}{p}+\frac{1}{4\epsilon^n}(1-|u|^2)^2dv_g.\label{1005}
\end{align}
Since $u|_{\partial D}\subseteq S^{n-1}$ is constant, then $|u|^2=1$ and for any $\eta \in T(\partial D)$, $du(\eta)=0$. Thus
\begin{align}
\int_{\partial D}S_{u,p}^{GL}(r\frac{\partial}{\partial r},\nu)ds_g&=\int_{\partial D}r\{\frac{|du|^p}{p}\langle \frac{\partial}{\partial r}, \nu \rangle-|du|^{p-2}\langle du(\frac{\partial}{\partial r}),du(\nu) \rangle\}ds_g. \nonumber  \\
&=\int_{\partial D}r\{\frac{|du|^p}{p}\langle \frac{\partial}{\partial r}, \nu \rangle-|du|^{p-2}\langle \frac{\partial}{\partial r},\nu \rangle |du|^2\}ds_g \nonumber \\
&=\int_{\partial D}r|du|^p\langle \frac{\partial}{\partial r},\nu\rangle\frac{1-p}{p}ds_g. \label{1007}
\end{align}
Note that $D$ is starlike, by (\ref{1006}) and (\ref{1007})
\begin{align}
\int_D\langle S^p_{GL},\frac{1}{2}L_Xg \rangle dv_g \leq 0.\label{1008}
\end{align}
From (\ref{1005}) and (\ref{1008}), we have
\begin{align}
\int_D\frac{|du|^p}{p}+\frac{1}{4\epsilon^n}(1-|u|^2)^2dv_g=0. \nonumber
\end{align}
Therefore $u$ is constant.
\end{proof}

\noindent Tian Chong\\
Department of Mathematics, Shanghai Second Polytechnic University\\
Shanghai 201209, China \\
E-mail: valery4619@sina.com \\

\noindent Bofeng Cheng\\
School of Mathematical Science, Fudan University,
Shanghai 200433, China.\\
E-mail: 13641817752@163.com\\

\noindent Yuxin Dong\\
 School of Mathematical Science, Fudan University, Shanghai, 200433,  China.\\
E-mail: yxdong@fudan.edu.cn\\

\noindent Wei Zhang\\
 School of Mathematics, South China University of Technology, Guangzhou, 510641, China.\\
E-mail: sczhangw@scut.edu.cn\\

\end{document}